\documentclass[10pt,a4paper,leqno]{article}

\usepackage{ifpdf}

\usepackage[T1]{fontenc}
\usepackage[italian,german,english]{babel}
\usepackage[tbtags]{amsmath}
\usepackage{amsthm}
\usepackage{amsfonts}
\usepackage{amssymb}
\usepackage{amscd}
\usepackage{eucal}
\usepackage{upref}

\usepackage{a4wide}

\renewcommand{\le}{\leqslant}
\renewcommand{\ge}{\geqslant}
\DeclareMathOperator{\Wig}{Wig}

\DeclareMathOperator{\im}{Im}
\renewcommand{\Im}{\im}

\newcommand{\abs}[1]{\left\lvert #1\right\rvert}
\newcommand{\norm}[1]{\left\lVert #1\right\rVert}
\newcommand{\bigabs}[1]{\bigl\lvert #1\bigr\rvert}

\newcommand{\Bigabs}[1]{\Bigl\lvert #1\Bigr\rvert}

\theoremstyle{definition}
\newtheorem{theorem}{Theorem}

\newtheorem{proposition}[theorem]{Proposition}
\newtheorem{corollary}[theorem]{Corollary}
\newtheorem{definition}[theorem]{Definition}

\theoremstyle{remark}
\newtheorem*{remark}{Remark}

\title{{\huge \bfseries
Regularity of a class \\  of  differential operators.
}}
\author{Ernesto Buzano and Alessandro Oliaro\\
\parbox{8cm}{\begin{center}\small Dipartimento di Matematica, Universit\`a di Torino \\
 Via Carlo Alberto 10, 10123 Torino, Italy
\end{center}}}
\date{}

\allowdisplaybreaks

\begin{document}

\selectlanguage{english}
\maketitle

\section{Introduction}

In this paper we deal with the problem of regularity for non
hypo-elliptic partial differential equations with polynomial
coefficients. An operator $A:\mathcal{S}^\prime\rightarrow
\mathcal{S}^\prime$ is regular if $u$ is a Schwartz function
whenever $Au$ is a Schwartz function. It is well known that
hypo-elliptic partial differential operators in the sense of
Definition 25.2 of \cite{Shubin} are regular. On the other
hand, the
problem of finding necessary and sufficient conditions for
the
regularity of a differential operator with polynomial coefficients is still  open.

In the case of
ordinary differential equations, in \cite{Nicola-Rodino} necessary and
sufficient
conditions for regularity are found under additional hypotheses. For partial
differential equations the problem is much more
complicated. We
refer for example to \cite{Wong}, where the regularity of the
twisted laplacian is proved, by  explicit
computation of the
heat kernel and Green function. The twisted laplacian
can be
viewed as a Schr\"odinger operator perturbed with
electromagnetic
field; it is intimately connected with the sub-laplacian on
the
Heisenberg group, see \cite{Thangavelu}, and is also
studied
from the point of view of spectral theory, see for example
\cite{Koch-Ricci}.

In this paper we follow a new approach, related to transformations
of Wigner type, to prove regularity of a wide class of non
hypo-elliptic partial differential operators with polynomial
coefficients. The approach consists in applying a Wigner-like
transform to a general partial differential equation. This idea is
already present in some works related to engineering applications,
cf. \cite{Galleani-Cohen-1}, \cite{Galleani-Cohen-2}. In these
papers some equations are analyzed, looking for the Wigner transform
of the solution. Instead of finding  first a solution $u$, and then
computing its Wigner transform $\Wig[u]$, the equation itself is
Wigner-transformed  obtaining an equation in $\Wig[u]$. In this way
it is possible to find,  in some cases, the exact expression of
$\Wig[u]$.

In the present paper we use tensor products of topological vector
spaces in order to apply the Wigner transform technique to the study
of regularity in the Schwartz space of partial differential
operators with polynomial coefficients.  We refer to Section
\ref{sec:1} for a precise statement of the result; the idea is that
a linear (hypo-elliptic) operator is associated to a linear non
hypo-elliptic one, and the Wigner-like transform allows to transfer
the regularity from an operator to the other one. In this way we
easily recover the regularity of the twisted laplacian, as well as
of generalized versions of it. Moreover, we analyze several other
examples, as general second order operators and operators associated
to a complete Newton polygon.

In this paper, we consider only partial differential equations in two variables.
There is no difficulty to generalize the results to an
arbitrary
even number of variables; we do not present our  results
in such a
generality to keep the  formalism as simple as possible.

The paper is organized as follows. In Section \ref{sec:3}
we give
general results concerning regularity and tensor product of
operators; in Section \ref{sec:1} we prove the main
result, and in
Section \ref{sec:4} we apply such a result in concrete cases,
by analyzing several examples.

\section{Regularity and extension of the variables.}\label{sec:3}

Given  a linear operator $A$ on $\mathcal S'(\mathbb R)$ such that
$A\bigl(\mathcal S(\mathbb R)\bigr)\subset L^2(\mathbb R)$, we
denote by $A^\ast$ the adjoint of $A_{\vert \mathcal S}$ with
respect to the inner product on $L^2(\mathbb R)$. The domain
$\mathcal D(A^\ast)$ of $A^\ast$ is given by all $u\in L^2(\mathbb
R)$ for which there exists $v\in L^2(\mathbb R)$ such that
\begin{equation*}
(A\phi\,\vert \, u)_{L^2}=(\phi\,\vert\,v)_{L^2},\qquad \phi\in \mathcal S.
\end{equation*}

\begin{proposition}
Consider a continuous linear operator $A$ on $\mathcal S'(\mathbb
R)$. If   $A\bigl(\mathcal S(\mathbb R)\bigr)\subset \mathcal
S(\mathbb R)$, then  $\mathcal S(\mathbb R)\subset \mathcal
D(A^\ast)$,   $A^{\ast\ast}$ is the closure of $A_{\vert \mathcal
S}$ in $L^2(\mathbb R)$ and $A_{\vert \mathcal S}$ is continuous on
$\mathcal S(\mathbb R)$.
\end{proposition}
\begin{proof}
For all $\psi\in \mathcal S(\mathbb R)$,
\begin{equation*}
u\mapsto  \langle A u\,\vert\,\overline\psi\rangle
\end{equation*}
is a continuous functional on $L^2(\mathbb R)$. Then there exists $v\in L^2(\mathbb R)$ such that
\begin{equation*}
\langle A u\,\vert\,\overline\psi\rangle=(u\,\vert\,v)_{L^2},\qquad u\in L^2(\mathbb R).
\end{equation*}
This implies that $\psi\in \mathcal D(A^\ast)$ and that $A^\ast \psi=v$. Since $\mathcal S(\mathbb R)$ is dense in $L^2(\mathbb R)$, it follows that $A^{\ast\ast}$ is a closed extension of $A_{\vert \mathcal S}$.

Consider a sequence  $(u_n)$  in $\mathcal S(\mathbb R)$ such that $u_n\to u$ and $A u_n \to v$ in $\mathcal S(\mathbb R)$. Then $u_n\to u$ and $A u_n\to v$ in $\mathcal S'(\mathbb R)$. But $A$ is continuous on $\mathcal S'(\mathbb R)$, hence $A u=v$. Thus $A$ is a closed operator on $\mathcal S(\mathbb R)$. Therefore $A$ is continuous on $\mathcal S(\mathbb R)$, by Closed Graph Theorem.
\end{proof}

\begin{definition}
A linear operator $A$ on $\mathcal S'(\mathbb R)$ is
\emph{regular} if
\begin{equation*}
Au\in  \mathcal S(\mathbb R) \implies u\in \mathcal S(\mathbb R),\qquad u\in\mathcal S'(\mathbb R).
\end{equation*}
\end{definition}

We denote by $E\widehat\otimes F$ the  \emph{topological} tensor products of two nuclear spaces $E$ and $F$.
Given two linear continuous operators $A_j:E_j\to F_j$, with $j\in\{1,2\}$, between nuclear spaces,
 $A_1\widehat\otimes A_2:E_1\widehat\otimes E_2\to F_1\widehat \otimes F_2$ is the unique continuous linear operator such that $A_1\otimes A_2(u_1\otimes u_2)=A_1(u_1)\otimes A_2(u_2)$, for all $(u_1,u_2)\in  E_1\times E_2$.

\smallskip
Given a linear operator $A$ on a vector space $E$, we denote by $\mathcal N(A)$ the subspace of the solutions to the equation $Au=0$.

\begin{theorem}\label{thm:9}
Consider a continuous linear operator $A$ on $\mathcal S'(\mathbb R)$ such that $A\bigl(\mathcal S(\mathbb R)\bigr)\subset \mathcal S(\mathbb R)$ and assume that
\begin{enumerate}
\item
$\mathcal N(A^\ast)\subset \mathcal S(\mathbb R)$,
\item\label{itm:3}
$\mathcal S(\mathbb R)=A\bigl(\mathcal S(\mathbb R)\bigr)\oplus
\mathcal N(A^\ast)$.
\item\label{itm:4}
$\mathcal S'(\mathbb R)=A\bigl(\mathcal S'(\mathbb R)\bigr)\oplus
\mathcal N(A^\ast)$,
\end{enumerate}
Let  $I$ be the identity operator on $\mathcal S'(\mathbb R)$.
Then $A$ is regular and one-to-one if and only if the tensor product $A\widehat \otimes I$ is regular on $\mathcal S'(\mathbb R^2)$.
\end{theorem}

\begin{remark}
 We say that the operator $A\widehat \otimes I$ is obtained by ``extension of the variables from $\mathbb R$ to $\mathbb R^2$''. This explains the title of this section.
\end{remark}

\begin{proof}
Assume $A\widehat \otimes I$ is regular.

Consider $u\in \mathcal S'(\mathbb R)$ such that $Au\in\mathcal
S(\mathbb R)$. Then $A\widehat \otimes I (u\otimes v)=A u\otimes
v\in \mathcal S(\mathbb R^2)$ for all $v\in \mathcal S(\mathbb R)$.
If $A\widehat\otimes I$ is regular, $u\otimes v$ must
belong to $\mathcal S(\mathbb R^2)$ for all $v\in \mathcal S(\mathbb
R)$. But this is impossible, unless $u$ belongs to $\mathcal
S(\mathbb R)$.

Assume now there exists $\phi\in\mathcal S(\mathbb R)\setminus
\{0\}$ such that $A\phi=0$. Let $\delta$ be the Dirac distribution,
then $\phi(x)\otimes \delta(y)$ belongs to the kernel of
$A\widehat\otimes I$, but not to $\mathcal S(\mathbb R^2)$, in
contradiction with the regularity of $A\widehat\otimes I$.

Assume now that $A$ is regular and one-to-one.
By assumption \eqref{itm:3} and Open Mapping Theorem $A_{\vert\mathcal S}$ is an isomorphism of $\mathcal S(\mathbb R)$ onto $A\bigl(\mathcal S(\mathbb R)\bigr)$. Then, by Propositions 43.7 and 43.9 of \cite{Treves}, $A_{\vert \mathcal S}\widehat \otimes I_{\vert \mathcal S}$ is an isomorphism of $\mathcal S(\mathbb R^2)$ onto
\begin{equation*}
A_{\vert \mathcal S}\widehat \otimes I_{\vert \mathcal S}(\mathcal S(\mathbb R^2))=A\bigl(\mathcal S(\mathbb R)\bigr)\widehat \otimes\mathcal S(\mathbb R).
\end{equation*}
 Since
\begin{equation*}
A_{\vert \mathcal S}\widehat \otimes I_{\vert \mathcal S}=(A\widehat \otimes I)_{\vert\mathcal S},
\end{equation*}
we have that $(A\widehat \otimes I)_{\vert\mathcal S}$ is an isomorphism of  $\mathcal S(\mathbb R^2)$ onto $A\bigl(\mathcal S(\mathbb R)\bigr)\widehat \otimes\mathcal S(\mathbb R)$.

Moreover, thanks to  hypothesis \eqref{itm:3}, we have:
\begin{equation}
\label{eqn:6}
\mathcal S(\mathbb R^2)=\Bigl(A\bigl(\mathcal S(\mathbb R)\bigl)\oplus \mathcal N(A^\ast)\Bigr)\widehat \otimes\mathcal S(\mathbb R)=
A \widehat \otimes I \bigl(\mathcal S(\mathbb R^2)\bigr)\oplus \mathcal N(A^\ast)\widehat\otimes\mathcal S(\mathbb R).
\end{equation}

By Proposition 3.17.2 of \cite{Horvath}, the Open Mapping Theorem is
true for continuous  linear maps from a Pt\'ak space onto a
barrelled Hausdorff space. Now by Proposition 3.17.6 of
\cite{Horvath}, $\mathcal S'(\mathbb R)$ is a Pt\'ak space. On the
other hand, from hypothesis \eqref{itm:4} we have that $A(\mathcal
S'\bigl(\mathbb R)\bigr)$ is canonically isomorphic to $\mathcal
S'(\mathbb R)/\mathcal N(A^\ast)$. Then, by Corollary (a) to
Proposition 3.6.4 of \cite{Horvath}, we have that $A\bigl(\mathcal
S'(\mathbb R)\bigr)$ is a barrelled Hausdorff space. Then, by Open
Mapping Theorem, $A$ is an isomorphism of $\mathcal S'(\mathbb R)$
onto $A\bigl(\mathcal S'(\mathbb R)\bigr)$.

Therefore from Propositions 43.7 and 43.9 of \cite{Treves}
 we obtain that $A\widehat\otimes I$ is an isomorphism of $\mathcal S'(\mathbb R^2)$ onto a dense subspace of  $A\bigl(\mathcal S'(\mathbb R)\bigr)\widehat\otimes\mathcal S'(\mathbb R)$.  Since  $\mathcal S'(\mathbb R^2)$ is complete, we have
 \begin{equation*}
 A\widehat\otimes I\bigl(\mathcal S'(\mathbb R^2)\bigr)=A\bigl(\mathcal S'(\mathbb R)\bigr)\widehat\otimes\mathcal S'(\mathbb R).
\end{equation*}

Moreover, since $\mathcal S'(\mathbb R)=A\bigl(\mathcal S'(\mathbb R)\bigr)\oplus \mathcal N(A^\ast)$, we obtain
\begin{equation}
\label{eqn:7}
\mathcal S'(\mathbb R^2)=
A\bigl(\mathcal S'(\mathbb R)\bigr)\widehat\otimes\mathcal S'(\mathbb R)\oplus \mathcal N(A^\ast)\otimes\mathcal S'(\mathbb R)=
A\widehat\otimes I(\mathcal S'(\mathbb R^2))\oplus \mathcal N(A^\ast)\widehat\otimes\mathcal S'(\mathbb R).
\end{equation}

Consider now $u\in\mathcal S'(\mathbb R^2)$ such that $f=(A\widehat \otimes I) u\in\mathcal S(\mathbb R^2)$. Since $f$ belongs to $\mathcal S(\mathbb R^2)$, thanks to \eqref{eqn:6}, there exist unique $v\in \mathcal S(\mathbb R^2)$ and $h\in \mathcal N(A^\ast)\widehat\otimes\mathcal S(\mathbb R)\subset \mathcal N(A^\ast)\widehat\otimes \mathcal S'(\mathbb R)$ such that
$(A\widehat\otimes I)u=(A\widehat\otimes I)v+h$. But then, \eqref{eqn:7} implies that $h=0$ and $u=v\in\mathcal S(\mathbb R^2)$.
\end{proof}

\section{Regularity of a class of differential operators.}\label{sec:1}

\begin{definition}
A polynomial  $a(x,\xi)$ on $\mathbb R\times\mathbb R$
is \emph{hypo-elliptic} if it does not vanish outside a compact set and
\begin{equation*}
\lim_{\abs{x}+\abs{\xi}\to\infty} \frac {\bigabs{\partial_x a(x,\xi)}+\bigabs{\partial_\xi a(x,\xi)}}{\bigabs{a(x,\xi)}}=0.
\end{equation*}
\end{definition}

\begin{theorem}\label{thm:4}
A differential operator  $A:\mathcal S'(\mathbb R)\to \mathcal S'(\mathbb R)$ with polynomial hypo-elliptic symbol
is regular and satisfies the hypotheses of Theorem \ref{thm:9}.
\end{theorem}
\begin{proof}
By Tarski-Seidenberg Theorem, see Appendix A of \cite{Hormander}, we have that $a$ is hypo-elliptic in the sense of Definition 25.2 of \cite{Shubin}. Then the result follows from Theorem 25.3 of  \cite{Shubin}.
\end{proof}
\begin{theorem}\label{thm:5}
Consider a  differential operator  $A:\mathcal S'(\mathbb R)\to \mathcal S'(\mathbb R)$ with polynomial hypo-elliptic symbol.
Then $A\widehat\otimes I$ is regular if and only if $A$ is one-to-one.
\end{theorem}
\begin{proof}
It follows from Theorems \ref{thm:9} and \ref{thm:4}.
\end{proof}

\begin{theorem}\label{thm:6}
Consider the differential operator on $\mathcal S'(\mathbb R^2)$
\begin{equation}\label{eqn:10}
B=\sum_{j+k\le m}c_{j,k} (x-q D_y)^j \bigl(y+p D_x\bigr)^k,
\end{equation}
where $p$ is a  real number and
\begin{equation*}
q=1-p.
\end{equation*}

Let $A$ be the differential operator on $\mathcal S'(\mathbb R)$ with symbol
\begin{equation*}
a(x,\xi)=\sum_{j+k\le m}c_{j,k}  x^j \xi^k.
\end{equation*}
 Assume the symbol $a$ be hypo-elliptic. Then $B$ is  regular if and only if $A$ is one-to-one.
\end{theorem}
\begin{proof}
Introduce the Wigner-like transform
 of a function $f\in\mathcal S(\mathbb R^2)$:
\begin{equation*}
\Wig_p[f](x,y)=
\frac 1{(2\pi)^{\nu/2}}\int e^{-i z y}f\bigl(x+(1-p)z,x-p z\bigr)\, d z.
\end{equation*}

Since $\Wig_p$ an isomorphism both on $\mathcal S(\mathbb R^2)$  and
on $\mathcal S'(\mathbb R^2)$, the result follows from Theorem \ref{thm:5} and the following identity.
\begin{equation}
\label{eqn:8}
B\Wig_p[w]=\Wig_p\bigl((A\widehat\otimes I) w\bigr),\qquad w\in \mathcal S(\mathbb R^2).
\end{equation}
We prove \eqref{eqn:8}, by induction. This means that  we may assume $m=1$.

Define the operators:
\begin{equation*}
M_1 w(x,y)= xw(x,y),\qquad M_2w(x,y)=yw(x,y),
\end{equation*}
and
\begin{equation*}
D_1 w(x,y)= D_xw(x,y),\qquad D_2w(x,y)=D_yw(x,y),
\end{equation*}
where, as usual,
\begin{equation*}
D_x=-i\partial_x,\qquad D_y=-i\partial_y.
\end{equation*}
Then we have
\begin{align}
& \label{eqn:1}\begin{aligned}[t]
&D_1\Wig_p[w](x,y)=\frac 1{(2\pi)^{\nu/2}}\int e^{-iy z} D_x\bigl(w(x+q z,x-p z)\bigr)\,d z
\\
&\qquad = \Wig_p[D_1 w](x,y)+\Wig_p[D_2 w](x,y),
\end{aligned}
\\ &\label{eqn:2}
\begin{aligned}[t]
&D_2\Wig_p[w](x,y)=\frac 1{(2\pi)^{\nu/2}}\int -z_j e^{-iy z} w(x+qz,x-pz)\,d z
\\
&\qquad=\frac 1{(2\pi)^{\nu/2}}\int  e^{-iy z} \bigl(x-p z\bigr)
w(x+qz,x-pz)\,d z
\\
&\qquad\quad-\frac 1{(2\pi)^{\nu/2}}\int  e^{-iy z} \bigl(x+qz\bigr)
w(x+qz,x-pz)\,d z
\\
 &\qquad=\Wig_p[M_2w](x,y)-\Wig_p[M_1w](x,y),
\end{aligned}
\\ &\label{eqn:3}
\begin{aligned}[t]
&M_1\Wig_p[w](x,y)=\frac 1{(2\pi)^{\nu/2}}\int e^{-iy z} x w(x+qz,x-pz)\,d z
\\
&\qquad=\frac 1{(2\pi)^{\nu/2}}\int  e^{-iy z} qx
w(x+qz,x-pz)\,d z
\\
&\qquad\quad+\frac 1{(2\pi)^{\nu/2}}\int  e^{-iy z} px
w(x+qz,x-pz)\,d z
\\
&\qquad=\frac 1{(2\pi)^{\nu/2}}\int  e^{-iy z} q(x-pz)
w(x+qz,x-pz)\,d z
\\
&\qquad\quad+\frac 1{(2\pi)^{\nu/2}}\int  e^{-iy z} p(x+qz)
w(x+qz,x-pz)\,d z
\\
&\qquad=  q\Wig_p[M_2 w](x,y)+p\Wig_p[M_1 w](x,y),
\end{aligned}
\\ &\label{eqn:4}
\begin{aligned}[t]
&M_2\Wig_p[w](x,y)=\frac 1{(2\pi)^{\nu/2}}\int \Bigl(-D_ze^{-iy z}\Bigr) w(x+qz,x-pz)\,d z
\\
&\qquad=\frac 1{(2\pi)^{\nu/2}} q\int  e^{-iy z}
(D_1 w)(x+q z,x-p z)\,d z
\\
&\qquad\quad
-\frac 1{(2\pi)^{\nu/2}} p\int  e^{-iy z}
(D_2 w)(x+q z,x-p z)\,d z
\\
&\qquad=  q\Wig_p[D_1 w](x,y)-p\Wig_p[D_2 w](x,y).
\end{aligned}
\end{align}
Then from \eqref{eqn:1} and \eqref{eqn:4} we obtain:
\begin{equation*}
\Wig_p[D_1 w]=
(M_2+pD_1) \Wig_p[w],
\end{equation*}
and from \eqref{eqn:2} and \eqref{eqn:3} we obtain:
\begin{equation*}
\Wig_p[M_1 w]
= (M_1-qD_2)\Wig_p[w]. \qedhere
\end{equation*}
\end{proof}

Observe that the symbol of the operator  \eqref{eqn:10} cannot be hypo-elliptic is the sense of the Definition 25.2 of \cite{Shubin}.
In fact we have the following proposition.

\begin{proposition}\label{pro:18}
Let $b(x,y;\xi,\eta)$ be the symbol of the operator
\eqref{eqn:10}, and
\begin{equation*}
\tilde{a}(x,\xi) = \sum_{j+k\leq m}\, \sum_{n\leq
\max\{j,k\}} c_{j,k} (iq)^n n!
\binom{j}{n} \binom{k}{n} x^{j-n}
\xi^{k-n}.
\end{equation*}
Then for all $(x_0,\xi_0)\in \mathbb R\times\mathbb R$, we have that
\begin{equation*}
b(x_0+q\eta,y_0-p\xi;\xi,\eta)=\tilde{a}
(x_0,y_0),\qquad (\xi,\eta)\in\mathbb R\times\mathbb R.
\end{equation*}
\end{proposition}
\begin{proof}
By Theorem 3.4 of
\cite{Shubin} we have that the symbol of the operator $B$
is given by
\begin{align*}
b(x,y;\xi,\eta) &=\sum_{j+k\leq m} c_{j,k}
\sum_{n\in\mathbb{Z}_+} \frac{(-i)^n}{n!}
\partial^n_\eta (x-q\eta)^j \partial^n_y
(y+p\xi)^k \\
&= \sum_{j+k\leq m} c_{j,k} \sum_{n\leq \max\{j,k\}}
(iq)^n  \binom{j}{n}
\binom{k}{n}n! (x-\eta+p\eta)^{j-n}
(y+p\xi)^{k-n}. \qedhere
\end{align*}
\end{proof}

We end this section with a simple observation which widens  the range of applicability of the Theorem \ref{thm:6}.

\begin{proposition}\label{pro:15}
Given a differential operator $B=b(z,D)$ on $\mathbb R^2$, with polynomial coefficients,
consider a $2\times 2$ non-singular real matrix $T$ and set
\begin{equation*}
B_T=a_T(z,D)=a(T' z,T^{-1}D).
\end{equation*}
where $T'$ is the transposed matrix. Then $B$ is regular if and only if $B_T$ is regular.
\end{proposition}
\begin{proof}
Consider $u\in\mathcal S'(\mathbb R^2)$. A simple computation yields
\begin{equation}
\label{eqn:29}
B_T(u\circ T')=(Bu)\circ T'.
\end{equation}
Since $u\circ T'$ belongs to $\mathcal S(\mathbb R^2)$ if and only if $u\in \mathcal S(\mathbb R^2)$, the result follows immediately from \eqref{eqn:29}.
\end{proof}

\section{Examples.}\label{sec:4}

\subsection{Second order self-adjoint operators.}\label{sec:5}

\begin{proposition}\label{pro:17}
Consider the  polynomial
\begin{equation}\label{eqn:34}
a(x,\xi)=a_2x^2+a_1 x+a_0-ib_1+2(b_1x+b_0)\xi+c_0\xi^2,
\end{equation}
where the coefficients $a_2,a_1,a_0,b_1,b_0,c_0$ are real numbers and $i^2=-1$.

Assume  there exist $r_1,r_0,s_1,s_0$ such that
\begin{equation}\label{eqn:40}
b_1=r_1s_1,\qquad b_0=r_0s_0
\end{equation}
and
\begin{equation}\label{eqn:41}
s_0^2+s_1^2\le c_0,\qquad a_1^2< 4(a_2-r_1^2)(a_0-r_0^2).
\end{equation}

Then the operator
\begin{equation*}
B=a_2 (x-qD_y)^2+a_1(x-qD_y)+a_0-ib_1+2\bigl(b_1(x-qD_y)+b_0\bigr)(y+pD_x)+c_0(y+pD_x)^2
\end{equation*}
is regular, for all $p\in \mathbb R$ and $q=1-p$.
\end{proposition}
\begin{proof}
From \eqref{eqn:40} and \eqref{eqn:41} we obtain that
\begin{equation*}
b_1^2=(r_1s_1)^2<a_2c_0,
\end{equation*}
therefore the quadratic form
\begin{equation*}
a_2x^2+2b_1x\xi+c_0\xi^2
\end{equation*}
is  positive-definite and this implies that the polynomial $a$ is hypo-elliptic.

A simple computation shows that the operator $A=a(x,D)$ is symmetric on $\mathcal S(\mathbb R)$. This implies that $(Au\,\vert\,u)$ is real for all $u\in\mathcal S(\mathbb R)$. Then, by using \eqref{eqn:40} and \eqref{eqn:41},  we have the estimate
\begin{equation}\label{eqn:42}
\begin{split}
(Au\,\vert u)_{L^2}&=\int \Bigl((a_2x^2+a_1x+a_0)\abs{u}^2-ib_1 \abs{u}^2+2(b_1x+b_0)Du\,\overline u+c_0D^2u\,\overline u\Bigr)d x
\\&=\int \Bigl((a_2x^2+a_1x+a_0)\abs{u}^2+(b_1x+b_0)(Du\,\overline u+u\overline {Du})+c_0 \abs{Du}^2\Bigr)d x
\\&\ge\int \Bigl(\bigl((a_2-r_1^2)x^2+a_1x+a_0-r_0^2\bigr)\abs{u}^2+(c_0-s_1^2-s_0^2) \abs{Du}^2\Bigr)d x
\\&\ge\frac {4(a_0-r_0^2)(a_2-r_1^2)-a_1^2}{a_2-r_1^2}\,\norm{u}_{L^2}^2.
\end{split}\end{equation}

Now, if $Au=0$, by regularity  $u$ belongs to $\mathcal S(\mathbb R)$, and estimate \eqref{eqn:42} together with   \eqref{eqn:41}  implies   $u=0$. This shows that $A$ is one-to-one. Then
the result follows from  Theorem \ref{thm:6}.
\end{proof}

In particular, when  $a_2=4$, $c_0=1/4$ $a_1=a_0=b_1=b_0=0$ and $p=q=1/2$, Proposition \ref{pro:17} implies the regularity of
\begin{equation}\label{eqn:44}
B=4 \Bigl(x-\frac 12 D_y\Bigr)^2+\frac 14 \Bigl(y+\frac 12
D_x\Bigr)^2.
\end{equation}
Then   Proposition \ref{pro:15} with
\begin{equation*}
T=\begin{bmatrix} \frac 14 & 0 \\ 0 & 1\end{bmatrix},
\end{equation*}
implies the regularity of the twisted laplacian, already studied in
\cite{Wong}:
\begin{equation}\label{eqn:45}
B_T=\Bigl( D_y-\frac{1}{2}x\Bigr)^2+\Bigl( D_x+\frac{1}{2}y\Bigr)^2
\end{equation}

In the next section we study some generalizations of this operator.

\subsection{Generalized twisted laplacian.}

\begin{proposition}\label{pro:6}
Fix positive integers $h$, $k$, $m$ and $n$, with
\begin{equation}\label{eqn:43}
m<h,\quad n<k, \quad nh+mk\geq hk.
\end{equation}
Let $\lambda,\sigma>0$, $\mu,\nu\ge 0$, $p\in\mathbb{R}$, and write
$q=1-p$. Then the operator
\begin{equation}\label{eqn:9}
\begin{split}
B = &\lambda (x-qD_y)^{2h} + \mu (x-qD_y)^m (y+pD_x)^{2n} (x-qD_y)^m
\\
&+ \nu (y+pD_x)^n (x-qD_y)^{2m} (y+pD_x)^n + \sigma (y+pD_x)^{2k}
\end{split}
\end{equation}
is  regular.
\end{proposition}

\begin{proof}
Consider the operator $A$ on $\mathcal{S}^\prime(\mathbb{R})$
defined by
\begin{equation*}
A = \lambda x^{2h} + \mu x^m D^{2n} x^m + \nu D^n x^{2m} D^n +
\sigma D^{2k},
\end{equation*}
and let $a(x,\xi)$ be its symbol. Conditions \eqref{eqn:43} imply that
 $a(x,\xi)$ is multi-quasi-elliptic, in the sense of definition of page 62 of
\cite{Boggiatto-Buzano-Rodino}. As a matter of facts, an easy computation shows that $a(x,\xi)$
 is  hypo-elliptic.

Let us prove that $A=a(x,D)$ is one-to-one. Given $u\in \mathcal S(\mathbb R)$ we have
\begin{equation*}\begin{split}
(A u\,\vert\,u)_{L^2} &=  \lambda (M^{2h} u\,\vert\, u)_{L^2} + \mu
(M^m D^{2n} M^m u\,\vert\, u)_{L^2} + \nu (D^n M^{2m} D^n u\,\vert\,
u)_{L^2} + \sigma (D^{2k} u\,\vert\,u)_{L^2}
\\
&=\lambda\norm{M^h u}_{L^2}^2 + \mu \norm{D^n M^m u}_{L^2}^2 + \nu
\norm{M^m D^n u}_{L^2}^2 + \sigma \norm{D^k u}_{L^2}^2 \ge 0,
\end{split}\end{equation*}
where $Mu(x)=xu(x)$ is the multiplication operator.
Then, $Au=0$ implies $M^h u=0$ and $D^k u=0$, that is $u=0$. So $A$
is one-to-one, and from Theorem \ref{thm:6} the operator $B$  is
regular.
\end{proof}

Observe that Proposition \ref{pro:18} implies that the symbol of the
operator \eqref{eqn:9} is not hypo-elliptic.

\bigskip
When  $h=k=1$, $\lambda=4$, $\sigma=1/4$, $\mu=\nu=0$ and $q=p=1/2$,
we recover the regularity of \eqref{eqn:44}, and consequently of
\eqref{eqn:45}, as in Section \ref{sec:5}. On the other hand, we can
now treat a quasi homogeneous twisted laplacian of higher order, with
arbitrary coefficients. Consider $\rho,\tau\in\mathbb{R}$ such that
\begin{equation*}
 \rho\tau\ne 0,\qquad
\rho\ne \tau,
\end{equation*}
and let $B$ the operator \eqref{eqn:9} with $\mu=\nu=0$,
$\lambda=(\rho-\tau)^{2h}$, $\sigma=1$, $p=\frac{\rho}{\rho-\tau}$.
Consider moreover
\begin{equation*}
T=\begin{bmatrix} \frac{\rho}{\rho-\tau} & 0 \\ 0 &
\tau\end{bmatrix}.
\end{equation*}
Then by Proposition \ref{pro:6} and Proposition \ref{pro:15} we have
that the operator
\begin{equation}\label{eqn:16}
B_T = (D_y+\rho x)^{2h} + (D_x+\tau y)^{2k}
\end{equation}
is regular.

\subsection{Weyl-Wick transform and positivity of operators.}\label{sec:2}

The problem of proving injectivity of an operator is in general non-trivial, since it is strictly connected to the knowledge of its
spectrum. On the other hand the (strict) positivity is a
sufficient condition for an operator to be one-to-one. In this
section we give a general method to construct hypo-elliptic
polynomial symbols of positive operators. Applying Theorem
\ref{thm:6} we then obtain  regular operators with
non hypo-elliptic symbol.

\begin{definition}
The \emph{Weyl-Wick transform} of the polynomial
\begin{equation}\label{eqn:13}
a(x,\xi)=\sum_{j+k\le m} c_{j,k} x^j \xi^k,\qquad (x,\xi)\in\mathbb R\times\mathbb R,
\end{equation}
 is the polynomial
\begin{equation}\label{eqn:30}
W[a](x,\xi) = \sum_{n\in\mathbb Z_+} \frac{(-1)^n}{4^n n!}\,
\Delta_{x,\xi}^n \sum_{l\in\mathbb Z_+} \frac{(-i)^{l}}{2^{l}l!}\,
\partial_x^l \partial_\xi^l\, a(x,\xi).
\end{equation}
\end{definition}

\begin{proposition}\label{pro:19}
The Weyl-Wick transform of a polynomial $a(x,\xi)$ of order $m$ is a
polynomial of order $m$ and $W$ is invertible on the space of
polynomials in the $(x,\xi)$ variables; in particular
\begin{equation*}
W^{-1}[a](x,\xi) = \sum_{l\in\mathbb{Z}_+} \frac{i^l}{2^l l!}
\partial_x^l\partial_\xi^l \sum_{n\in\mathbb{Z}_+} \frac{1}{4^n n!}
\Delta_{x,\xi}^n a(x,\xi),
\end{equation*}
\end{proposition}

\begin{proof}
The simple proof is left for the reader, see Problem 24.11 and Theorem 23.3 of
\cite{Shubin}.
\end{proof}

Consider now the function
\begin{equation*}
\Phi_{y,\eta}(x) = \pi^{-1/4} e^{ix\eta} e^{-\frac{1}{2} (y-x)^2},
\end{equation*}
where $y,\eta\in\mathbb R$ are parameters, and the corresponding orthogonal projection in $L^2(\mathbb R)$ on $\Phi_{y,\eta}$:
\begin{equation}
\label{eqn:33}
P_{y,\eta} u(x) = \Bigl(\int u(t) \overline{\Phi_{y,\eta}(t)}\,dt\Bigr) \Phi_{y,\eta}(x),
\end{equation}
for $u\in L^2(\mathbb R)$. We have the following result.

\begin{proposition}\label{pro:7}
Let $A$ be the differential operator  with symbol \eqref{eqn:13}, then we have
\begin{equation}
\label{eqn:32} Au =\frac {1}{2\pi} \int W[a](y,\eta)
(P_{y,\eta}u)\,dy\,d\eta,\qquad u\in\mathcal S(\mathbb R).
\end{equation}
\end{proposition}

\begin{proof}
It follows from  Problem 24.10 and Theorem
23.3 of \cite{Shubin}.
\end{proof}

\begin{proposition}\label{pro:16}
If $a$ is a polynomial such that $W[a](x,\xi)> 0$ for almost all
$(x,\xi)\in\mathbb{R}^2$, then the operator $A=a(x,D)$ is
one-to-one. \end{proposition}
\begin{proof}
It follows from \eqref{eqn:32} and  \eqref{eqn:33} that
\begin{equation*}
(Au\,\vert\, u)_{L^2}=\frac 1{2\pi} \int W[a](y,\eta) \Bigabs{\int
u(x)\overline {\Phi_{y,\eta}(x)}\,dx}^2d y d \eta.
\end{equation*}
Then, if $W[a]>0$ almost everywhere, $Au=0$ implies that
\begin{equation*}
\pi^{-1/4} \int e^{-ix\eta} e^{-\frac{1}{2}(y-x)^2}u(x)\,d x=  \int u(x)\overline {\Phi_{y,\eta}(x)}\,dx=0,\qquad (y,\eta)\in\mathbb R\times\mathbb R,
\end{equation*}
and therefore
\begin{equation*}
e^{-\frac{1}{2}(y-x)^2}u(x)=0,\qquad (x,y)\in\mathbb R\times\mathbb R,
\end{equation*}
that is $u=0$.
\end{proof}

Since the Weyl-Wick transform preserves hypo-ellipticity, we have the
following result.

\begin{corollary}\label{cor:2}
Let $a(x,\xi)$ be an hypo-elliptic polynomial such that $a(x,\xi)>
0$ for almost all  $(x,\xi)\in\mathbb R\times \mathbb R$, and let
$r(x,\xi)=W^{-1}[a](x,\xi)$. Write
\begin{equation*}
r(x,\xi)=\sum_{j+k\le m} c_{j,k} x^j \xi^k,
\end{equation*}
where $m\in\mathbb{Z}_+$ and $c_{j,k}\in\mathbb{R}$, $j+k\leq m$ are
uniquely determined by $a$. Then the operator
\begin{equation*}
B=\sum_{j+k\le m} c_{j,k} (x-qD_y)^j(y+p D_x)^k,
\end{equation*}
with $q=1-p$, is  regular.
\end{corollary}

\subsection{Operators with negative index.}

Now we consider differential operators of first order.

\begin{proposition}\label{pro:9}
Fix $\alpha\in\mathbb{C}$ and a positive integer $m$. Suppose that $(\Im \alpha)^m>0$. Then for every $p\in\mathbb{R}$ the operator in $\mathbb{R}^2$
\begin{equation*}
y+p D_x + \alpha \bigl( x+(p-1)D_y\bigr)^m
\end{equation*}
is  regular.
\end{proposition}

\begin{proof}
Consider the operator $A$ with symbol
\begin{equation}\label{eqn:46}
a(x,\xi) = \xi+\alpha x^m.
\end{equation}
Thanks to Theorem \ref{thm:6}, it is enough to prove that $a$ is  hypo-elliptic and $A$ is one-to-one. The condition $(\Im \alpha)^m>0$ implies  the  hypo-ellipticity of $a$.

From Theorem \ref{thm:4} the kernel of $A$ is then contained in $\mathcal{S}(\mathbb{R})$. On the other hand the classical solutions of $Au=0$ are the functions
\begin{equation*}
u(x) = \beta e^{-i\alpha \frac{x^{m+1}}{m+1}},\quad \beta\in\mathbb{C},
\end{equation*}
which do not belong to $\mathcal{S}(\mathbb{R})$ for $(\Im \alpha)^m>0$. This implies that $A$ is one-to-one, and  the proof is complete.
\end{proof}

In view of the hypotheses of Theorem \ref{thm:9}, it is interesting to observe that
when  $m$ is odd and $\Im \alpha>$ is positive,  the operator $A=D+\alpha x^m$ has index
 $-1$, hence in particular $\mathcal N(A^\ast)\not = 0$.

\end{document}